\documentclass[oneside,english]{amsart}
\usepackage[T1]{fontenc}
\usepackage[latin9]{inputenc}
\usepackage{geometry}
\usepackage{mathtools}
\usepackage{amstext}
\usepackage{amsthm}
\usepackage{amssymb}

\makeatletter
\numberwithin{equation}{section}
\numberwithin{figure}{section}
\theoremstyle{plain}
\newtheorem{thm}{\protect\theoremname}
  \theoremstyle{definition}
  \newtheorem{defn}[thm]{\protect\definitionname}
  \theoremstyle{plain}
  \newtheorem{lem}[thm]{\protect\lemmaname}
  \theoremstyle{plain}
  \newtheorem{prop}[thm]{\protect\propositionname}

\DeclareMathOperator{\slicerank}{slice-rank}

\makeatother

\usepackage{babel}
  \providecommand{\definitionname}{Definition}
  \providecommand{\lemmaname}{Lemma}
  \providecommand{\propositionname}{Proposition}
\providecommand{\theoremname}{Theorem}

\usepackage{hyperref}
\PassOptionsToPackage{hyphens}{url}\usepackage{hyperref}

\providecommand{\MR}{\relax\ifhmode\unskip\space\fi MR }

\newcommand\arXiv[1]{arXiv:\href{http://arXiv.org/abs/#1}{#1}}

\begin{document}

\title{Monochromatic Equilateral Triangles in the Unit Distance Graph}

\author{Eric Naslund}
\date{\today}
\begin{abstract}
Let $\chi_{\Delta}(\mathbb{R}^{n})$ denote the minimum number of
colors needed to color $\mathbb{R}^{n}$ so that there will not be
a monochromatic equilateral triangle with side length $1$. Using
the slice rank method, we reprove a result of Frankl and Rodl, and
show that $\chi_{\Delta}\left(\mathbb{R}^{n}\right)$ grows exponentially
with $n$. This technique substantially improves upon the best known
quantitative lower bounds for $\chi_{\Delta}\left(\mathbb{R}^{n}\right)$,
and we obtain 
\[
\chi_{\Delta}\left(\mathbb{R}^{n}\right)>(1.01446+o(1))^{n}.
\]
\end{abstract}

\maketitle

\section{Introduction}

The chromatic number of Euclidean space, $\chi\left(\mathbb{R}^{n}\right)$,
is the the minimum number of colors required to color $\mathbb{R}^{n}$
such that no two points at distance $1$ have the same color. When
$n=2$ determining $\chi\left(\mathbb{R}^{2}\right)$ is known as
the Hadwiger-Nelson problem \cite{Gardner1960BookofProblems,Hadwiger1944,Hadwiger1945},
and the best existing bounds are 
\[
5\leq\chi\left(\mathbb{R}^{2}\right)\leq7,
\]
where the lower bound is a recent improvement due to De Grey \cite{DeGrey2018ChromaticNumberOfThePlane}.
For large $n$, Larman and Rogers \cite{LarmanRogers1972DistancesWithinEuclideanSpaces}
proved that 
\[
\chi\left(\mathbb{R}^{n}\right)\leq\left(3+o(1)\right)^{n},
\]
and conjectured that the rate of growth of this function is exponential.
This was confirmed by Frankl and Wilson in 1981, who applied the linear
algebraic method to prove that $\chi\left(\mathbb{R}\right)^{n}\geq\left(1+c+o(1)\right)^{n}$
where $c=\frac{\sqrt{2}-1}{2}=0.207\dots$. The best existing lower
bound is due to Raigorodskii \cite{Raigorodskii2000OnTheChromaticNumberOfASpace}
\[
\chi\left(\mathbb{R}^{n}\right)>(1.239\dots+o(1))^{n},
\]
and is proven using a refined application of the linear algebraic
method.

In Euclidean Ramsey theory, a configuration of points $S\subset\mathbb{R}^{k}$
is said to be \emph{exponentially Ramsey }if we need exponentially
many colors as a function of $n$ to guarantee that there will be
no monochromatic copies of $S$ in any coloring of $\mathbb{R}^{n}$.
In 1987 Frankl and Rodl proved that the $k$-simplex is exponentially
Ramsey for every $k$ \cite[Theorem 1.18]{FranklRodl1987FranklRodlTheorem}.
Specifically, they proved that for any $k$, there exists $\epsilon_{k}>0$
such that any coloring of $\mathbb{R}^{n}$ with less than $(1+\epsilon_{k}+o(1))^{n}$
colors contains a monochromatic regular $k$-simplex with side length
$1$. In this paper we will examine the specific case $k=3$, and
we let $\chi_{\Delta}\left(\mathbb{R}^{n}\right)$ denote the minimum
number of colors needed to color $\mathbb{R}^{n}$ so that it does
not contain a monochromatic equilateral triangle of side lengths $1$.
The best existing lower bound is due to Sagdeev \cite{Sagdeev2018AnImprovedFranklRodlEquilTriangle},
who refined Frankl and Rodl's theorem on intersections of families
of sets to prove that
\[
\chi_{\Delta}\left(\mathbb{R}^{n}\right)>(1+c+o(1))^{n}
\]
where 
\[
c=0.00085.
\]
In this paper, we use the Slice-Rank method introduced in \cite{TaosBlogCapsets}
to give the following quantitative improvement to Sagdeev's result:
\begin{thm}
\label{thm:monochromatic_triangles}We have that 
\[
\chi_{\Delta}\left(\mathbb{R}^{n}\right)>(1+c+o(1))^{n}
\]
where 
\[
c=0.01446\dots 
\]
\end{thm}

\section{The Slice Rank}

In a breakthrough result, Croot, Lev, and Pach \cite{CrootLevPachZ4},
introduced a powerful new way to apply the polynomial method, and
Ellenberg and Gijswijt \cite{EllenbergGijswijtCapsets} used their
technique to prove that the largest capset\footnote{A \emph{cap set }is a subset of $\mathbb{F}_{3}^{n}$ whose only solutions
to $x+y+z=0$ are the trivial solutions $x=y=z$. } in $\mathbb{F}_{3}^{n}$ has size bounded by $O\left(2.756^{n}\right)$,
settling a longstanding conjecture. Tao symmetrized their argument
in his blog \cite{TaosBlogCapsets}, and introduced the notion of
the slice rank of a tensor:
\begin{defn}
Let $X,Y,Z$, be finite sets, and let $\mathbb{F}$ be a field. The
slice rank of a function 
\[
F\colon X\times Y\times Z\rightarrow\mathbb{F}
\]
is the smallest $k$ such that 
\[
F(x,y,z)=\sum_{i=1}^{a}f_{i}(x,y)g(z)+\sum_{i=a+1}^{b}f_{i}(x,z)g(y)+\sum_{i=b+1}^{k}f_{i}(y,z)g(x).
\]
\end{defn}
The slice rank method has seen a wide array of applications \cite{BlasiakChurchCohnGrochowNaslundSawinUmans2016MatrixMultiplication,BurgisserGargOliveiraWalterWigderson2017NullConeSliceRankTheorem,GeShangguan2017NoRightAngles,Naslund2017EGZ,Naslund2017Partition,NaslundSawinSunflower,Sawin2017NonAbelianMatchings},
and we refer the reader to section 4 of \cite{BlasiakChurchCohnGrochowNaslundSawinUmans2016MatrixMultiplication}
for an in depth discussion of the properties of the slice rank. For
our purposes, we will need the critical lemma, which was proven by
Tao:
\begin{lem}
\label{lem:critical_lemma}Let $X$ be a finite set, and let $X^{n}$
denote the $n$-fold Cartesian product of $X$ with itself. Suppose
that 
\[
F\colon X^{n}\rightarrow\mathbb{F}
\]
is a diagonal tensor, that is 
\[
F(x_{1},\dots,x_{n})=\sum_{a\in A}c_{a}\delta_{a}(x_{1})\cdots\delta_{a}(x_{n})
\]
for some $A\subset X$, $c_{a}\neq0$, where 
\[
\delta_{a}(x)=\begin{cases}
1 & x=a\\
0 & \text{otherwise}
\end{cases}.
\]
Then 
\[
\slicerank(F)=|A|.
\]
\end{lem}
\begin{proof}
See \cite[Lemma 1]{TaosBlogCapsets} or \cite[Lemma 4.7]{BlasiakChurchCohnGrochowNaslundSawinUmans2016MatrixMultiplication}.
\end{proof}

\section{Main Result}

We will deduce Theorem \ref{thm:monochromatic_triangles} the following
proposition:
\begin{prop}
\label{prop:equil_triangle_prop}For $k\leq\frac{n}{2}$, let $S\subset\{0,1\}^{n}$
be the set of elements with exactly $k$ ones, and let $p$ be the
smallest odd prime such that $p>\frac{k}{4}$. Suppose that $A\subset S$
does not contain $x,y,z$ with 
\[
\|x-y\|_{2}=\|y-z\|_{2}=\|z-x\|_{2}=\sqrt{2p}.
\]
Let $\epsilon_{0}=n^{0.525}$ denote an error term. Then for sufficiently
large $n$ 
\[
|A|\leq3\cdot\min_{0<t<1}\frac{(1+t)^{n}}{t^{\frac{n}{3}+\frac{k}{6}+\epsilon_{0}}}.
\]
\end{prop}
\begin{proof}
Baker, Harman, and Pintz's \cite{BakerHarmanPintz2001DifferenceBetweenConsecutivePrimes}
bounds for the largest prime gap imply that for sufficiently large
$n$ 
\[
p<\frac{k}{4}+\epsilon_{0}.
\]
For $x,y,z\in S$ consider the polynomial 
\[
F\colon S\times S\times S\rightarrow\mathbb{F}_{p}
\]
defined by 

\[
F(x,y,z)=\prod_{i=1}^{n}\left(x_{i}+y_{i}+z_{i}-1\right).
\]
If $x,y,z$ satisfy $F(x,y,z)\neq0$, that is if there is no $i$
such that $x_{i}+y_{i}+z_{i}=1$, then we must have $\|x-y\|_{2}=\|y-z\|_{2}=\|z-x\|_{2}$,
and so they form an equilateral triangle. Furthermore, if $F(x,y,z)\neq0$,
then we can upper bound the distance 
\[
\|x-y\|_{2}^{2}<2p.
\]
To see this, for each $j\in\{0,1,2,3\}$ let $a_{j}=\#\left\{ i:\ x_{i}+y_{i}+z_{i}=j\right\} $,
and note that $a_{1}=0$ since $F(x,y,z)\neq0$. Since there are $n$
coordinates, and $3k$ total entries equal to $1$, we have that 
\[
a_{0}+a_{2}+a_{3}=n\ \ \text{ and }\ \ 2\cdot a_{2}+3\cdot a_{3}=3k.
\]
Subtracting 3 times the first equation from the second, we obtain
\[
a_{2}=3n-3k-3a_{0}.
\]
The only coordinates that contribute to the distance are counted by
$a_{2}$, and so 
\[
\|y-z\|_{2}^{2}+\|z-x\|_{2}^{2}+\|x-y\|_{2}^{2}=2a_{2}.
\]
Hence 
\[
\|x-y\|_{2}^{2}=2n-2k-2a_{0}.
\]
The smallest $a_{0}$ can be is if $a_{3}=0$ and all $3k$ ones are
used by coordinates where the sum is $2$. That is, $a_{0}\geq n-\frac{3k}{2},$
and hence
\begin{equation}
\frac{\|x-y\|_{2}^{2}}{2}\leq\frac{k}{2}<2p.\label{eq:triangle_distance_bound}
\end{equation}

Let $G\colon S\times S\rightarrow\mathbb{F}_{p}$ be given by
\[
G(x,y)=\left(1-\left(\frac{\|x-y\|_{2}^{2}}{2}\right)^{p-1}\right),
\]
and note that $\frac{1}{2}\|x-y\|_{2}^{2}$ will always be an integer
for $x,y\in S$. If if $x\neq y$ are such that $\frac{1}{2}\|x-y\|_{2}^{2}<2p$,
then $G(x,y)\neq0$ if and only if $\frac{1}{2}\|x-y\|_{2}^{2}=p$.
For $x,y,z\in S$ define 
\[
H(x,y,z)\coloneqq F(x,y,z)G(x,y).
\]
This function will be non-zero when $x=y=z$, and will be zero whenever
$x,y,z$ do not form an equilateral triangle with side length $\sqrt{2p}$.
Suppose that $A\subset S$ contains no equilateral triangles of side
length $\sqrt{2p}$. Then $H$ restricted to $A\times A\times A$
will be a diagonal tensor with non-zero diagonal elements, and so
by Lemma \ref{lem:critical_lemma} 
\[
|A|\leq\slicerank(H).
\]
The polynomial $H$ will have degree at most $n+2p<n+\frac{k}{2}+\epsilon_{0}$,
and we may expand it as a linear combination of monomials of the form
\[
\left(x_{1}^{d_{1}}\cdots x_{n}^{d_{n}}\right)\left(y_{1}^{e_{1}}\cdots x_{n}^{e_{n}}\right)\left(z_{1}^{f_{1}}\cdots z_{n}^{f_{n}}\right)
\]
where $e_{i},d_{i},f_{i}\in\{0,1\}$ for each $i$, and 
\[
\left(\sum_{i=1}^{n}d_{i}\right)+\left(\sum_{i=1}^{n}e_{i}\right)+\left(\sum_{i=1}^{n}f_{i}\right)\leq n+\frac{k}{2}+\epsilon_{0}.
\]
For each monomial, one of these sums will be at most $\frac{1}{3}(n+\frac{k}{2}+\epsilon_{0})$,
and hence by always slicing off the lowest degree piece we have
\[
\slicerank(H)\leq3\cdot\#\left\{ v\in\{0,1\}^{n}:\ \sum_{i=1}^{n}v_{i}\leq\frac{n}{3}+\frac{k}{6}+\frac{\epsilon_{0}}{3}\right\} .
\]
For any $0<t<1$, 
\[
\#\left\{ v\in\{0,1\}^{n}:\ \sum_{i=1}^{n}v_{i}\leq r\right\} =\sum_{k\leq r}\binom{n}{k}\leq t^{-r}\sum_{k=0}^{n}\binom{n}{k}t^{k}
\]
since the coefficient $t^{k-r}$ will be greater than $1$ for $k\leq r$.
Taking the minimum over $t$, for $r=\frac{n}{3}+\frac{k}{6}+\frac{\epsilon_{0}}{3}$,
we obtain the stated result.
\end{proof}
\begin{proof}[Proof of Theorem \ref{thm:monochromatic_triangles}]
 Let $S\subset\{0,1\}^{n}$ be the subset of vectors with exactly
$k$ ones, for $k\leq\frac{n}{2}$, and let $A\subset S$ be the largest
subset that does not contain an equilateral triangle of side length
$\sqrt{2p}$. Then we need at least $\frac{|S|}{|A|}$ sets that do
not contain an equilateral triangle of side lengths $\sqrt{2p}$ to
cover $S$. Rescale every point in $\mathbb{R}^{n}$ by a factor of
$\sqrt{2p}$ so that these points are at distance $1$. As $\frac{1}{\sqrt{2p}}S\subset\mathbb{R}^{n}$,
it follows that 
\[
\chi_{\Delta}(\mathbb{R}^{n})\geq\frac{|S|}{|A|},
\]
and by Proposition \ref{prop:equil_triangle_prop} 
\[
\chi_{\Delta}(\mathbb{R}^{n})\geq\frac{1}{3}\binom{n}{k}\max_{0<t<1}\frac{t^{\frac{n}{3}+\frac{k}{6}+\epsilon_{0}}}{(1+t)^{n}}.
\]
Since this bound holds for any $0\le k\leq\frac{n}{2}$, we may take
the maximum and write
\[
\chi_{\Delta}(\mathbb{R}^{n})\geq\frac{1}{3}\max_{0<t<1}\left[\left(\frac{t^{\frac{1}{3}+\frac{\epsilon_{0}}{n}}}{1+t}\right)^{n}\max_{0\leq k\leq\frac{n}{2}}\binom{n}{k}t^{\frac{k}{6}}\right].
\]
Expanding $(1+x)^{n}$, we have that for any $0<x<1$ 
\[
\frac{(1+x)^{n}}{n+1}<\max_{0\leq k\leq\frac{n}{2}}\binom{n}{k}x^{k}<(1+x)^{n},
\]
and hence 
\[
\max_{0\leq k\leq\frac{n}{2}}\binom{n}{k}t^{\frac{k}{6}}>\frac{1}{n+1}\left(1+t^{\frac{1}{6}}\right)^{n}.
\]
We must have $t^{\frac{1}{3}}>\frac{1}{2}$, since otherwise the value of the function we are maximizing will be less than $1$. Hence $t$ will be bounded away from $0$, which implies that $t^{\frac{\epsilon_{0}}{n}}=1+o(1)$.
Simplifying the result, we obtain 
\[
\chi_{\Delta}(\mathbb{R}^{n})>\left(\max_{0<t<1}\frac{t^{\frac{1}{3}}\left(1+t^{\frac{1}{6}}\right)}{1+t}+o(1)\right)^{n}\]
and the desired bound follows by computing the maximum.
\end{proof}

\specialsection*{Acknowledgements}

I would like to thank Yufei Zhao for several helpful conversations.
This work was partially supported by the NSERC PGS-D scholarship,
and by Ben Green's ERC Starting Grant 279438, Approximate Algebraic
Structure and Applications.

\bibliographystyle{siam}

\end{document}